\pdfoutput=1
\documentclass[10pt]{article}
\usepackage[a4paper, margin=1in]{geometry}
\usepackage{changepage}
\usepackage{amsmath, amssymb, amsthm}
\usepackage[utf8]{inputenc}
\usepackage[english]{babel}
\usepackage[numbers]{natbib}
\usepackage{url}
\usepackage[usenames]{color}
\usepackage[colorlinks=true]{hyperref}
\usepackage{cleveref}
\usepackage{tabularx}
\usepackage{breqn}
\usepackage[title]{appendix}
\usepackage{xcolor}
\theoremstyle{plain}

\newtheorem{corollary}{Corollary}[section]
\newtheorem{lemma}{Lemma}[section]

\newtheorem{theorem}{Theorem}[section]
\newtheorem*{theorem*}{Theorem}
\newtheorem{remark}{Remark}[section]

\theoremstyle{definition}
\newtheorem{example}{Example}[section]
\crefname{conjecture}{Conjecture}{Conjectures}
\crefname{theorem}{Theorem}{Theorems}
\crefname{theorem*}{Theorem}{Theorems}
\crefname{corollary}{Corollary}{Corollaries}
\crefname{lemma}{Lemma}{Lemmas}
\crefname{proposition}{Proposition}{Propositions}
\crefname{remark}{Remark}{Remarks}
\crefname{note}{Note}{Notes}
\crefname{definition}{Definition}{Definitions}
\crefname{notation}{Notation}{Notations}
\crefname{example}{Example}{Examples}
\crefname{question}{Question}{Questions}
\crefname{section}{\S}{Sections}
\crefname{equation}{Equation}{Equations}
\crefformat{equation}{(eq.~#2#1#3)} 
\DeclareMathOperator\sign{sign}

\newcommand{\floor}[1]{\left\lfloor #1 \right\rfloor}
\newcommand{\ceil}[1]{\left\lceil #1 \right\rceil}

\newcommand{\sgn}{\operatorname{sgn}}

\newcommand{\Z}{\mathbb{Z}}

\newcommand{\seqnum}[1]{\href{https://oeis.org/#1}{\rm \underline{#1}}}

\title{On modular representations of C-recursive integer sequences}
\author{Mihai Prunescu \footnote{Research Center for Logic, Optimization and Security (LOS), Faculty of Mathematics and Computer Science, University of Bucharest, Academiei 14, Bucharest (RO-010014), Romania; Institute of Logic and Data Science, Bucharest, Romania; Simion Stoilow Institute of Mathematics of the Romanian Academy, Research unit 5, P. O. Box 1-764, Bucharest (RO-014700), Romania. E-mail: {\tt mihai.prunescu@imar.ro}, {\tt mihai.prunescu@gmail.com}.},
Joseph M. Shunia \footnote{Veeam Software, Columbus, OH, USA (Remote). E-mail: {\tt jshunia@gmail.com}}}
\date{February 24, 2025}

\begin{document}

\maketitle

\begin{abstract} \noindent
Prunescu and Sauras-Altuzarra showed that all C-recursive sequences of natural numbers have an arithmetic div-mod representation that can be derived from their generating function. This representation consists of computing the quotient of two exponential polynomials and taking the remainder of the result modulo a third exponential polynomial, and works for all integers $n \geq 1$. Using a different approach, Prunescu proved the existence of two other representations, one of which is the mod-mod representation, consisting of two successive remainder computations. This result has two weaknesses: (i) the representation works only ultimately, and (ii) a correction term must be added to the first exponential polynomial. We show that a mod-mod representation without inner correction term holds for all integers $n \geq 1$. This follows directly from the div-mod representation by an arithmetic short-cut from outside. 
\\[2mm]
\end{abstract}

\section{Introduction}\label{SectionIntroduction}

The \textbf{C-recursive sequences of order $d$} are sequences $t : \mathbb N \rightarrow \mathbb C$ satisfying a relation of recurrence with constant coefficients:
$$
t(n+d) + a_1 t(n+d-1) + \dots + a_{d-1} t(n+1) + a_d t(n) = 0
$$
for all $n \in \mathbb{N}$. With the recurrence rule, we associate the polynomial:
$$
B(x) := 1 + a_1 X + \dots + a_d X^d .
$$

According with Theorem 4.1.1 in \cite{stanley}, and with Theorem 1 in \cite{PetkovsekZakrajsek}, the C-recursive sequences are exactly the sequences $(s(n))$ such that the generating function:
$$f(z) = \sum_{n \geq 0} t(n) z^n$$
is a rational function $A(z)/B(z)$ with $\deg(A) < \deg(B) = d$. 

Prunescu and Sauras-Altuzarra proved in \cite{prunescusauras2024representationcrecursive} that if $f$ is the generating function of a sequence consisting of natural numbers only, then
\begin{align*}
\exists c \in \mathbb{N} : \forall n \in \Z^+, \quad 
t(n) = \floor{ c^{n^2} f(c^{-n}) } \bmod c^n .    
\end{align*}
The exact statement will be given below, see Theorem \ref{ThmExtraction1}. When the sequence is C-recurrent, as $f$ is a rational function $A(x) / B(x)$ with $A(x), B(x) \in \mathbb Z[x]$, free terms of $A(x)$ and $B(x)$ positive, and
with $\deg(A) < \deg(B) = d$, one has the following {\bf div-mod representation}:
\begin{align*}
t(n) = \floor{ \frac{c^{n^2 + dn} A(c^{-n})}{c^{dn} B(c^{-n})} } \bmod c^n = 
\floor{ \frac{c^{n^2} \tilde A(c^{n})}{ \tilde B(c^n)} } \bmod c^n.    
\end{align*}  

Prunescu proved in \cite{prunescu2024representationscrecursive} the following \textbf{mod-mod representation}. Let $\alpha_d \not= 0$ be the  free term of the recurrence rule for a C-recursive sequence. Then there are $c, n_0 \in \mathbb N$ such that  for $n \geq n_0$ one has
\begin{align*}
t(n) = \frac{\left( \left( c^{n(d-1)+\ceil{n/2}} - \sgn(\alpha_d) \cdot  c^{n^2} \tilde A(c^n) \right) \bmod \tilde B(c^n) \right) \bmod c^n}{|\alpha_d|}
\end{align*}

This representation has the advantage that $(c^{n(d-1)+\ceil{n/2}} - \sgn(\alpha_d) \cdot  c^{n^2}\tilde A(c^n)) \bmod \tilde B(c^n)$ can be faster computed by modular arithmetic, but it has two disadvantages:
\begin{enumerate}
    \item[(i)] It needs in general a correction term $c^{n(d-1)+\ceil{n/2}}$ in the inner-most exponential polynomial.
    \item[(ii)] It holds only ultimately, for $n \geq $ some an a priori undetermined $n_0 \in \mathbb{N}$. In fact such an $n_0$ can be determined by conditions used in the proof, but it is not clear that one can always take $n_0 = 1$. 
\end{enumerate}
In the following we show that this result can be improved to: 
{\it If the free term of $\tilde B(c^n)$ is $\alpha_d \neq 0$, there exists $e \geq c$ such that the following holds for all $n \geq 1$:}
\begin{align*}
t(n) = \frac{-1 -\sign(\alpha_d)}{2} + \frac{1}{|\alpha_d|} \left( \left( (- \sign (\alpha_d) \cdot  e^{n^2} \tilde A(e^n))  \bmod \tilde  B(e^n) \right)  \bmod e^n \right).
\end{align*} 

Moreover, this improvement is done by a short-cut of modular arithmetic. The result becomes only a corollary to the div-mod representation. 

Other possibilities and methods to represent C-recursive sequences can be found in the monograph of Kauers and Paule \cite{kauerspaule}.

\section{Technical preparations}\label{SectionPrerequisites} 
We define $\mathbb{N}$ as the set of natural numbers with $0$ included.

The first three useful results refer to C-recursive sequences. The next Lemma can be also found in Everest and al \cite{everestandal}.

\begin{lemma}\label{LemmaGeneralInequality} [Prunescu \& Sauras-Altuzarra, Lemma 4, \cite{prunescusauras2024representationcrecursive}]
If $ t : \mathbb{N} \rightarrow {\mathbb{C}} $ is C-recursive, then there is an integer $ g \geq 1 $ such that $ | t(n) | < g^{n+1} $ for every integer $ n \geq 0 $. \end{lemma}

\begin{theorem}\label{ThmExtraction1} [Prunescu \& Sauras-Altuzarra, Theorem 1 and 2, \cite{prunescusauras2024representationcrecursive}] If $ t : \mathbb N \rightarrow \mathbb{N} $, $f(z)$ is its generating function, $ R $ is the radius of convergence of $ f(z) $ at zero and $ c $, $ m $ and $ n $ are three integers such that $ c \geq 2 $, $ n \geq m \geq 2 $, $ c^{- m} < R $ and $ t ( n ) < c^{n - 2} $ for every integer $ n \geq m $, then $$t(n) = \floor{ c^{n^2} f ( c^{- n} ) } \bmod c^n . $$ 
Also, if $ c \geq 8 $, $c^{-1} < R$ and $ t ( n ) < c^{n / 3} $ for every  $ n \geq 1 $, then the representation works for every $n \geq 1$. 
\end{theorem}

The following Corollary follows easily from \cref{ThmExtraction1}.

\begin{corollary}\label{CorBiggerBase}
If the representation stated in \cref{ThmExtraction1} holds for some $c \in \mathbb N$ for all $n \geq m$, then it holds also if we replace $c$ with any $d \geq c$ for all $n \geq m$.
\end{corollary}

The next three Lemmas are easy remarks of modular arithmetic. 

\begin{lemma} \label{lemma1}
If $B \geq 2$, $A \geq 1$, $B \nmid A$, then $- \floor{-A/B} = \floor{A/B} + 1$.
\end{lemma}
\begin{lemma} \label{lemma2}
If $a, y \in \mathbb{N}$ and $0 \leq (a y) < C$, then $\left( (ay) \bmod C\right) = a ( y \bmod C )$.
\end{lemma}
\begin{lemma} \label{lemma3}
If $x \not\equiv (C-1) \pmod{C}$ then $(x+1) \bmod C = (x \bmod C) + 1$.
\end{lemma} 

The next Lemma is the principal tool of this note.

\begin{lemma} \label{LemmaMain}
Let $A, B, C \in \mathbb{Z}$ such that $A, B > 0$, $C \geq 2$, $C \,|\,A$, $B \nmid A$, $B \bmod C \equiv a \mod C$, $a \not= 0$ and $|a| (\floor{A / B} \bmod C) < C$ if $a < 0$ respectively $a + a (\floor{A / B} \bmod C) < C$ if $a > 0$.
\begin{enumerate}
    \item[(i)] If $a < 0$, then
    \begin{align*}
        (A \bmod B) \bmod C = |a| (\floor{A / B} \bmod C) .
    \end{align*}
    \item[(ii)] If $a > 0$, then
    \begin{align*}
     \left( (-A) \bmod B \right) \bmod C = a (1 + \floor{A/B} \bmod C ) . 
    \end{align*}
\end{enumerate}
\end{lemma}
\begin{proof}
For both cases below we apply the Lemmas \ref{lemma1} and \ref{lemma2}. For the second case we apply also \cref{lemma3}. We proceed with the cases.
\begin{enumerate}
\item[(i)] $a < 0$:
\begin{align*}
(A \bmod B) \bmod C
&= \left( A \bmod C - (B \bmod C) ( \floor{A / B} \bmod C) \right) \bmod C \\
&= \left( 0 - (- |a| ) ( \floor{A / B} \bmod C) \right) \bmod C \\ 
&= \left( |a|  ( \floor{A / B} \bmod C) \right) \bmod C \\
&= |a| (  \floor{A / B} \bmod C )
\end{align*}
\item[(ii)] $a > 0$:
\begin{align*}
\left( (-A) \bmod B \right) \bmod C
&= \left( (-A) \bmod C - (B \bmod C) ( \floor{A / B} \bmod C) \right) \bmod C \\
&= \left( 0 - a ( \floor{(-A) / B} \bmod C) \right) \bmod C \\ 
&= ( a \cdot ( - \floor{(-A) / B} \bmod C ) ) \bmod C \\ 
&= ( a \cdot (1 +  \floor{A / B} \bmod C ) ) \bmod C \\
&= ( a + a  \floor{A / B} \bmod C ) \bmod C \\
&= a ( 1 +  \floor{A / B} \bmod C ) .
\end{align*}
\end{enumerate}
\end{proof}

\section{Applications to C-recursive sequences}\label{SectionApplications}

\begin{theorem}\label{TheoApplication}
Suppose that a sequence $t : \mathbb N \rightarrow \mathbb N$ has the following representation for $n \geq 1$:
\begin{align*}
t(n) = \floor{\frac{c^{n^2}\tilde A(c^n)}{\tilde B(c^n)}} \bmod c^n .
\end{align*}

If the free term of $\tilde B(c^n)$ is $\alpha_d < 0$, then there is some natural number $e$ such that for all $n \geq 1$:
\begin{align*}
t(n) =  \frac{1}{|\alpha_d|} \left(\left((e^{n^2} \tilde A(e^n)) \bmod \tilde B(e^n) \right) \bmod e^n \right).     
\end{align*}
If the free term of $\tilde B(c^n)$ is $\alpha_d > 0$, then there is some natural number $e$ such that for all  $n \geq 1$:
\begin{align*}
t(n) = -1 + \frac{1}{\alpha_d} \left( \left( (- e^{n^2} \tilde A(e^n))  \bmod \tilde  B(e^n) \right)  \bmod e^n \right).
\end{align*} 
\end{theorem}

\begin{proof}
    In order to apply apply Lemma \ref{LemmaMain}, we recall that according to Corollary \ref{CorBiggerBase} there is a $c_0 \in \mathbb N$ such that for all $e \geq c_0$ and
    for all $n \geq 1$,
    \begin{align*}
t(n) = \floor{\frac{e^{n^2}\tilde A(e^n)}{\tilde B(e^n)}} \bmod e^n .
\end{align*}
Let $e \geq c_0$ be a natural number to fulfill also other conditions, which will be stated below. We introduce the notations $A = e^{n^2}\tilde A(e^n)$, $B = \tilde B(e^n)$ and $C = e^n$ and we prove that for a good choice of $e$, these numbers fulfill the conditions of Lemma \ref{LemmaMain}.

The conditions  $C \,|\,A$ and $B \nmid A$ are always fulfilled. The conditions $A, B > 0$, $C \geq 2$ are fulfilled for all $n \geq 1$ if $d$ is sufficiently large, as the main coefficients of $\tilde A(x)$ and $\tilde B(x)$ are positive. Let $\alpha_d$ be the free term of the polynomial $\tilde B(x) = x^d B(\frac{1}{x})$, meaning that $a = a_d$. For $d$ sufficiently large, $B \bmod C \equiv \alpha_d \mod C$. By definition, $\alpha_d \neq 0$ because $\deg B = d$. We have to show that for $e$ sufficiently large,  $|\alpha_d| (\floor{A / B} \bmod C) < C$ if $\alpha_d < 0$ respectively $\alpha_d + \alpha_d (\floor{A / B} \bmod C) < C$ if $\alpha_d > 0$, for all $n \geq 1$. We recall that $\floor{A/B} \bmod C = t(n)$. 

If $\alpha_d < 0$, the inequality $|\alpha_d| (\floor{A / B} \bmod C) < C$ reads $|\alpha_d| t(n) < e^n$ and must be true for $n \geq 1$.

If $\alpha_d > 0$, the inequality $\alpha_d + \alpha_d (\floor{A / B} \bmod C) < C$ reads $\alpha_d (t(n) + 1) < e^n$ and must be true for $n \geq 1$.

But we know from \ref{LemmaGeneralInequality} that if $s: \mathbb N \rightarrow \mathbb C$ is a C-recursive sequence then there exists an $g \in \mathbb N$ such that for all $n \in \mathbb N$, $|s(n)| < g^{n+1}$. As $s(n) = 2|\alpha_d| t(n)$ is itself a C-recursive sequence and its values are natural numbers, there is some positive $g \in \mathbb N$ such that $2|\alpha_n| t(n) < g^{n+1}$ for all $n \geq 1$. We can always find an $e \in \mathbb N$ such that $g^{n+1} < e^n$ for all $n \geq 1$. 

If $t(n) \geq 1$, as $2|\alpha_d| t(n) > |\alpha_d| t(n)$ in the first case, respectively $2 |\alpha_d| t(n) \geq |\alpha_d|(1 + t(n)) $ in the second case, the conditions of Lemma \ref{LemmaMain} are fulfilled. If $t(n) = 0$, the conditions are fulfilled for every $e \geq 1$ in the first case and for every $e > \alpha_d$ in the second case. 

Finally we choose $e$ sufficiently large such that $e \geq c_0$ and all the conditions above are fulfilled. 
\end{proof} 

Below we put both cases in only one formula:

\begin{corollary}\label{CorOneFormula}
    Suppose that a sequence $t : \mathbb N \rightarrow \mathbb N$ has the following representation for $n \geq 1$:
\begin{align*}
t(n) = \floor{\frac{c^{n^2}\tilde A(c^n)}{\tilde B(c^n)}} \bmod c^n .
\end{align*} 
If the free term of $\tilde B(c^n)$ is $\alpha_d \neq 0$, there exists $e \geq c$ such that the following holds for all $n \geq 1$:
\begin{align*}
t(n) = \frac{-1 -\sign(\alpha_d)}{2} + \frac{1}{|\alpha_d|} \left( \left( (- \sign (\alpha_d) \cdot  e^{n^2} \tilde A(e^n))  \bmod \tilde  B(e^n) \right)  \bmod e^n \right).
\end{align*} 
\end{corollary}

\begin{remark}\label{oscillant} \rm In \cite{prunescu2024representationscrecursive} there is also an application to C-recursive sequences $t : \mathbb N \rightarrow \mathbb Z$. According to Lemma \ref{LemmaGeneralInequality},  for such a sequence there is a $h \in \mathbb N$ such that for all $n \geq 0$, $|t(n)| < h^{n+1}$. The sequence $s(n) = t(n) + h^{n+1}$ has values in $\mathbb N$ and is also C-recursive. Indeed, the generating function of $s$ is a rational function and can be computed by: 
$$f(z) + \frac{h}{1-hz},$$
where $f(z)$ is the generating function of the sequence $t$. Consequently, the sequence $s$ has a representation according to Corollary \ref{CorOneFormula}, while the sequence $t$ has the representation $t(n) = s(n) - h^{n+1}$.
\end{remark} 

In \cite{prunescu2024representationscrecursive}, also the following {\bf mod-div representation} is proved: {\it There exist $c$ and $n_0 \in \mathbb N$ such that:}
 $$\forall \, n\geq n_0\,\,\,\, t(n) = \left \lfloor 
    \frac{\left ( c^{n(d-2) + \left \lceil n/2 \right \rceil} + c^{n^2} \tilde A(b^n ) \right ) \bmod \tilde B(c^n)}{c^{n(d-1)}} \right \rfloor.$$
{\bf Open problem}: Is it possible to find a purely arithmetic trick which shows that the mod-div representation (or even an improved version) is only a corollary of the div-mod representation, in a similar way as the short-cut shown here? 


\section{Examples}\label{SectionExamples}

In this section, all representations are derived from the respective representations in \cite{prunescusauras2024representationcrecursive}. In some cases a larger exponentiation base $e > c$ is needed in order to keep the representation true for all $n \geq 1$. 

\subsection{Degree 2, natural numbers, negative free term}

The first group of examples consists of sequences of order $2$ with $a < 0$. As in \cite{prunescu2024representationscrecursive} it was shown that this kind of sequence does not need inner correction term, these representations are not different from the representations displayed there. 

\begin{example} \label{ExFibonacci}
[\textbf{Fibonacci numbers}, OEIS \seqnum{A000045}]
$$
\forall n \in \Z^+, \quad 
s(n) = \left( 3^{n^2 + n} \bmod (3^{2n} - 3^n - 1) \right) \bmod 3^n .
$$
\end{example}

\begin{example} \label{PropFormulaLucas}
[\textbf{Lucas numbers}, OEIS \seqnum{A000032}] The div-mod representation works for $c = 3$, see Prunescu and Sauras-Altuzarra \cite{prunescusauras2024representationcrecursive}. The mod-mod representation works for $e = 5$:
$$
\forall n \in \Z^+, \quad 
s(n) = \left( ( 2 \cdot 5^{n^2 + 2n} - 5^{n^2 + n}) \bmod (5^{2n} - 5^n - 1) \right) \bmod 5^n .
$$
\end{example}

\begin{example} \label{PropFormulaPellNumbers}
[\textbf{Pell numbers}, OEIS \seqnum{A000129}]
$$
\forall n \in \Z^+, \quad 
s(n) = \left( 3^{n^2 + n} \bmod (3^{2n} - 2 \cdot 3^n - 1) \right) \bmod 3^n .
$$
\end{example}

\begin{example} \label{PropFormulaPellLucasNumbers}
[\textbf{Pell-Lucas numbers}, OEIS \seqnum{A002203}]
$$
\forall n \in \Z^+, \quad 
s(n) = \left( ( 2 \cdot 9^{n^2 + 2n} - 2 \cdot 9^{n^2 + n}) \bmod (9^{2n} - 2 \cdot 9^n - 1) \right) \bmod 9^n .
$$
\end{example}

\subsection{Degree 2, natural numbers, positive free term} 

\begin{example}\label{PropFormulaNaturals}[\textbf{Natural numbers}, OEIS \seqnum{A001477}] 
$$
\forall n \in \Z^+, \quad
\left( \left( (- 4^{n^2 + n} ) \bmod (4^{2n} - 2 \cdot 4^n + 1) \right) \bmod 4^n \right) - 1 = n .
$$
\end{example}
 
\begin{example}\label{PropFormulaAllTwos} [\textbf{All-twos}, OEIS \seqnum{A007395}]
$$
\forall n \in \Z^+, \quad
\left( \left( (-2 \cdot 4^{n^2 + 2n} + 2 \cdot 4^{n^2 + n} ) \bmod ( 4^{2n} - 2  \cdot 4^n + 1  ) \right) \bmod 4^n \right) -1 = 2 .
$$
\end{example}

\begin{example}\label{PropFormulaMersenne} [\textbf{Mersenne numbers}, OEIS \seqnum{A000225}] 
$$
\forall n \in \Z^+, \quad
\frac{1}{2} \cdot \left ( \left( (- 6 ^ {n ^ 2 + n} ) \bmod (6 ^ {2n} - 3 \cdot 6 ^ n + 2 ) \right) \bmod 6^n \right ) -1 = 2^n - 1 .
$$
\end{example} 

\begin{example} \label{PropFormula2tonplusone} [$2^n + 1$, OEIS \seqnum{A000051}] The div-mod representation works for $c = 6$, see Prunescu and Sauras-Altuzarra \cite{prunescusauras2024representationcrecursive}. The mod-mod representation works for $e = 9$:
$$
\forall n \in \Z^+, \quad 
\frac{1}{2} \cdot  \left ( (  - 2 \cdot  9^{n^2 + 2n} + 3 \cdot  9^{n^2 + n} ) \bmod ( 9^{2n}  -3 \cdot  9^n + 2 ) \bmod 9^n \right ) -1 = 2^n + 1 .
$$
\end{example}

\begin{example} \label{PropA001081} [OEIS \seqnum{A001081}, OEIS \seqnum{A001080}] Consider Pell's equation:
\begin{align} \label{EqPells}
X^2 - k Y^2 = 1 .
\end{align}
The sequence of solutions $(x(n), y(n))$ with $(x(0), y(0)) = (1,0)$ are known to be C-recursive sequences, see \cite{prunescusauras2024representationcrecursive}. It is proved there that the sequences $(x(n))$ and $(y(n))$ are C-recursive and enjoy the following representations:
$$x(n) = \left \lfloor \frac{b^{n^2 + 2n} - x(1) b^{n^2 + n}}{b^{2n} - 2 x(1) b^n + 1} \right \rfloor  \bmod b^n,$$
$$y(n) = \left \lfloor \frac{y(1) b^{n^2 + n}}{b^{2n} - 2 x(1) b^n + 1} \right \rfloor  \bmod b^n.$$

For $k = 7$, the fundamental solution is $(x(1), y(1)) = (8,3)$. If $n \in \Z^+$, then 
\begin{align*}
x(n) &= \left( \left( (-  143 ^ {n ^ 2 + 2n} +8 \cdot 143 ^ {n ^ 2 + n} ) \bmod(143 ^ {2n} -16 \cdot 143 ^ n + 1 ) \right) \bmod 143 ^ n \right) - 1 , \\
y(n) &= \left( \left( ( -3 \cdot 64 ^ {n ^ 2 + n}  ) \bmod (64 ^ {2n} -16 \cdot 64 ^ n + 1 ) \right) \bmod 64 ^ n\right) - 1 .
\end{align*}
\end{example}

\subsection{Degree 2, integers, negative free term}

In this subsection, we obtain formulas for two Lucas sequences that take positive and negative values. They are computed according to Remark \ref{oscillant}.

\begin{example} (\href{https://oeis.org/A088137}{\texttt{OEIS A088137}}, \textbf{generalized Gaussian Fibonacci integers}) If $n \in \Z^+$, then:
$$
t(n) = \frac{1}{9}\left( \left (3 \cdot 91^{n^2 + 3n} - 5 \cdot 91^{n^2 + 2n} + 6 \cdot 91^{n^2 + n}) \bmod ({91^{3n} - 5\cdot 91^{2n} + 9 \cdot 91^n - 9}) \right ) \bmod 91^n \right ) $$ $$ - 3^{n+1}.
$$
The div-mod representation works for $c = 32$, see Prunescu and Sauras-Altuzarra \cite{prunescusauras2024representationcrecursive}. The mod-mod representation works for $e = 91$, which is a spectacular difference. 
\end{example}

\begin{example} (\href{https://oeis.org/A002249}{\texttt{OEIS A002249}}) If $n \in \Z^+$, then:
$$s(n) =\frac{1}{4}\left ( \left( (4 \cdot 21^{n^2 + 3n} - 7 \cdot 21^{n^2 + 2n} + 6 \cdot 21^{n^2 + n}) \bmod (21^{3n} - 3\cdot 21^{2n} + 4 \cdot 21^n - 4) \right ) \bmod 21^n \right ) $$ $$ - 2^{n+1}.$$
\end{example} 

The div-mod representation works for $c = 8$, see Prunescu and Sauras-Altuzarra \cite{prunescusauras2024representationcrecursive}. The mod-mod representation works for $e = 21$, so we remark again a big difference. 

\subsection{Degree 3, natural numbers, negative free term}
We finally apply the theory to some C-recursive natural sequences of degree three, whose recursions do not contain positive coefficients. Consequently, these representations do not need correction terms.

\begin{example} (\href{https://oeis.org/A000073}{\texttt{OEIS A000073}}, \textbf{Tribonacci numbers}) If $n \in \Z^+$, then:
$$ s(n) = \left ( \left  (2^{n^2 + n} \bmod (2^{3n} - 2^{2n} - 2^n - 1) \right ) \bmod 2^n \right ) . $$
\end{example}

\begin{example} (\href{https://oeis.org/A000931}{\texttt{OEIS A000931}}, \textbf{Padovan numbers}) If $n \in \Z^+$, then:
$$ s(n) =  \left ( (2^{n^2 +3 n} - 2^{n^2 + n} ) \bmod (2^{3n} - 2^n - 1) \right ) \bmod 2^n . $$ 
\end{example}

\begin{example} (\href{https://oeis.org/A000930}{\texttt{OEIS A000930}},  \textbf{Narayana's cows sequence}) If $ n \in \Z^+$, then:
$$ s(n) = \left ( 2^{n^2 +3 n}  \bmod (2^{3n} - 2^{2n} - 1) \right ) \bmod 2^n .$$
\end{example}


\end{document}